\documentclass{amsart}
\usepackage[english]{babel}
\usepackage{amsmath}
\usepackage{amsfonts}
\usepackage{amsthm}
\usepackage{inputenc}
\usepackage{amssymb}
\usepackage{graphicx}

\newtheorem{thm}{Theorem}[section]

\newtheorem{lem}[thm]{Lemma}
\newtheorem{prop}[thm]{Proposition}
\newtheorem{defn}[thm]{Definition}

\newtheorem{rem}[thm]{Remark}
\numberwithin{equation}{section}

\begin{document}

\title[The classification of algebras of level two]{The classification of algebras of level two}%

\author{A.Kh. Khudoyberdiyev}

\address{[A.\ Kh.\ Khudoyberdiyev] Institute of Mathematics,
National University of Uzbekistan, Tashkent, 100125, Uzbekistan.}
\email{khabror@mail.ru}

\begin{abstract}
This paper is devoted to the description of complex
finite-dimensional algebras of level two. We obtain the
classification  of algebras of level two in the varieties of
Jordan, Lie and associative algebras.
\end{abstract}

\maketitle



\section{Introduction}

Important subjects playing a relevant role in Mathematics and
Physics are degenerations, contractions and deformations of
algebras.


Degenerations of non-associative algebras were the subject of
numerous papers (see for instance \cite{B2, Casas, GrunHall,
Kashuba1} and references given therein), and their research
continues actively.

The general linear group $GL(V )$ over a field $\mathbb{K}$ acts
on the vector space $V^{*}\otimes V^{*}\otimes
 V,$ the space of $\mathbb{K}$-algebra structures, by the
change of basis. For two $\mathbb{K}$-algebra structures $\lambda$
and $\mu$ we say that $\mu$ is a degeneration of $\lambda$ if
$\mu$ lies in the orbit closure of $\lambda$ with respect to Zariski topology (it is denoted by $\mu \rightarrow \lambda $).
The orbit closure problem in this form is about the classification
of all degenerations of a certain algebra structure of a fixed
dimension. This problem also depends on a complete classification
of the corresponding algebra structures. Both problems are highly
complicated even in small dimensions.

It is known that if ground field is $\mathbb{C}$, then closures of orbits in Zariski
and standard topologies coincide. Therefore, mainly the degenerations of complex objects are investigated.

It is well-known the closed relations between associative, Lie and
Jordan algebras. In fact, commutator product defined on
associative algebra gives us Lie algebra while symmetrized product
gives Jordan algebra. Moreover, any Lie algebra is isomorphic to a
subalgebra of a certain commutator algebra. The analogue of this
result is not true for Jordan algebras, that is, there are Jordan
algebras which can not be obtained from symmetrized product on
associative algebras.

For complex Lie algebras we have description of degenerations up to dimension four and for nilpotent ones up to
dimension six \cite{B1, Seeley}. In the case of Jordan algebras we have the description of degenerations up to dimension four \cite{Kashuba2}.

Since any $n$-dimensional algebra degenerates to the abelian
algebra (denoted by $\mathfrak{a}_n$), the lowest edges end on
$\mathfrak{a}_n.$ In \cite{Gorb1} Gorbatsevich described the
nearest-neighbor algebras to $\mathfrak{a}_n$ (algebras of level
one) in the degeneration graphs of commutative and skew-symmetric
algebras. In the work \cite{Khud}
 it was ameliorated and
correcting of some non-accuracies made in \cite{Gorb1}. Namely, a
complete list of algebras level one in the variety of
finite-dimensional complex algebras is obtained.

In fact, Gorbatsevich studied in \cite{Gorb2} a very interesting
notion closely related to degeneration: $\lambda\to\mu$ (algebras
$\lambda$ and $\mu$ not necessarily have the same dimension) if
$\lambda \oplus \mathfrak{a}_{k}$ degenerates to $\mu \oplus
\mathfrak{a}_{m}$ in the sense considered in this paper for some
suitable $k, m \geq 0.$ The corresponding first three levels of
such type of degenerations are completely classified in
\cite{Gorb2}.

In this paper we study the description of finite-dimensional
algebras of level two over the field of complex numbers. More
precisely, we obtain the classification of algebras of level two
in the varieties of Jordan, Lie and associative algebras.

In the multiplication table of an algebra omitted products are
assumed to be zero. Moreover,  due to commutatively and
anticommutatively of Jordan and Lie algebras, symmetric products for these algebras
are also omitted.

\section{Preliminaries.}

In this section we give some basic notions and concepts used
through the paper.

Let $\lambda$ be a $n$-dimensional algebra. We know that the
algebra $\lambda$ may be considered as an element  of the affine
variety $Hom(V\otimes V,V)$ via the  mapping $\lambda \colon
V\otimes V\to V$ over a field $\mathbb{K}$. The linear reductive group
$GL_n(\mathbb{K})$ acts on the variety of $n$-dimensional algebra $Alg_n$
via change of basis, i.e.,
\[(g*\lambda)(x,y)=g\Big(\lambda\big(g^{-1}(x),g^{-1}(y)\big)\Big), \quad  g \in GL_n(\mathbb{K}), \  \lambda \in Alg_n .\]

The orbits $Orb(-)$ under this action are the  isomorphism classes
of algebras. Note that solvable (respectively, nilpotent) algebras
of the same dimension also  form an invariant subvariety of the
variety of algebras under the mentioned action.

\begin{defn} \label{def1} An algebra $\lambda$ is said to
degenerate to an algebra $\mu$, if $Orb(\mu)$ lies in the Zariski
closure of $Orb(\lambda)$. We denote this by $\lambda\to\mu$.
\end{defn}

The degeneration $\lambda \rightarrow \mu $ is called {\it
trivial}, if $\lambda$ is isomorphic to $\mu.$ Non-trivial
degeneration $\lambda \rightarrow \mu $ is called {\it direct
degeneration} if there is no chain of non-trivial degenerations of
the form: $\lambda \rightarrow \nu \rightarrow \mu$.

\begin{defn} \label{def2}
The level of a $n$-dimesional algebra $\lambda$ is
the maximum length of a chain of direct degenerations, which, of
course, ends with the algebra $\mathfrak{a_n}$ (the algebra with
zero multiplication).
\end{defn}

Here we give the description of the algebras of level one.

\begin{thm}  \cite{Khud} \label{th level one} A $n$-dimensional $(n \geq 3)$ algebra is algebra of level one if and only if it is isomorphic to one of the following algebras:
\[\begin{array}{rlll}p_n: & e_1e_i = e_i,
& e_ie_1 = - e_i, & 2 \leq i \leq n; \\
 n_3 \oplus
\mathfrak{a}_{n-3}: & e_1e_2= e_3, & e_2e_1 = - e_3;\\
\lambda_2 \oplus \mathfrak{a}_{n-2}: &  e_1e_1 =e_2; \\
\nu_n(\alpha): & e_1e_1= e_1, & e_1e_i = \alpha e_i, & e_ie_1 =
(1-\alpha) e_i, \  2 \leq i \leq n, \ \alpha\in\mathbb{C}.
\end{array}\]
\end{thm}

Note that algebras $\lambda_2 \oplus \mathfrak{a}_{n-2}$ and $\nu_n\big(\frac 1 2\big)$ are Jordan algebras.

It is remarkable that the notion of degeneration considered in
\cite{Gorb2} is weaker than notions which are used in this paper.
For instance, the levels by Gorbatsevich's work of the algebras
$p_n$ and $\nu_n(\alpha)$ do not equal to one, because of $p_n
\oplus \mathfrak{a}_{1} \to n_3 \oplus \mathfrak{a}_{n-2}$ and
$\nu_n(\alpha) \oplus \mathfrak{a}_{1} \to \lambda_2 \oplus
\mathfrak{a}_{n-1}.$

It is known that any finite-dimensional associative (Jordan) algebra
 $A$ is decomposed into a semidirect sum of semi-simple subalgebra $A_{ss}$
and nilpotent radical $Rad(A).$ Moreover, an arbitrary finite-dimensional semi-simple
associative (Jordan) algebra contains an identity element. Therefore, one can assume that a
finite-dimensional associative (Jordan) algebra over a field
$\mathbb{K}$ of $char \mathbb{K} = 0$ is either nilpotent or has
an idempotent element.

One of the important results of theory of associative algebras
related with idempotents is Pierce's decomposition. Let $A$ be an
associative algebra which contains an idempotent element $e$. Then
we have decomposition
$$A = A_{1,1} \oplus A_{1,0} \oplus A_{0,1} \oplus A_{0,0}$$ with property
$A_{i,j}\cdot A_{k,l} \subseteq  \delta_{j,k}A_{i,l},$ where
$\delta_{j,k}$ are Kronecker symbols. The subspaces $A_{j,k}$ are
called Pierce's components.

Below we present an analogue of Pierce's decomposition for Jordan algebras.

\begin{thm} \label{th Jacobson}\cite{Jacobson} Let $e$ be an idempotent of a Jordan algebra $J.$ Then we have the
following decomposition into a direct sum of subspaces
$$J  = P_0\oplus P_{\frac 1 2} \oplus P_1,$$
where $P_i  = \{x \in J \  |  \ x \cdot e = i x\}, \ i = 0; \frac 1
2; 1$  and the multiplications for the components $P_i$ are defined as follows:
$$P_1^2 \subseteq P_1, \quad P_1 \cdot P_0 =0, \quad P_0^2 \subseteq P_0,\quad P_0 \cdot P_{\frac 1 2} \subseteq P_{\frac 1 2}, \quad P_1 \cdot P_{\frac 1 2} \subseteq P_{\frac 1 2},
 \quad P_{\frac 1 2}^2 \subseteq P_0\oplus P_1.$$
\end{thm}

\section{Main result}

This section is devoted to the classifications of algebras of level two in the varieties of complex $n$-dimensional Jordan, Lie and associative algebras.

\subsection{Jordan algebras of level two}
In this subsection we give the classification of algebras of level
two in the variety of complex $n$-dimensional Jordan algebras.

\begin{thm}\label{thJordan} A $n$-dimensional $(n \geq 3)$ Jordan algebra is algebra of
level two if and only if it is isomorphic to one of the following
algebras:
$$\begin{array}{rllll}J_1= \{e\}\oplus \mathfrak{a}_{n-1}: & e \cdot e=e;\\[1mm]
J_2 = \{e_1, e_2, e_3, \dots, e_{n}\}: & e_1 \cdot e_1=e_1, & e_1
\cdot
e_{i}=e_i,  & 2 \leq i \leq n;\\[1mm]
J_3= \{e_1, e_2, e_{3}\}\oplus \mathfrak{a}_{n-3}: & e_1 \cdot
e_2=e_3. \end{array}$$ \end{thm}

\begin{proof} Firstly we suppose that semi-simple part of the Jordan algebra $J$
is non-trivial, i.e., $J_{ss} \neq 0.$ Thereby, there exists an
unite element $e$ of $J_{ss}$ and $J$  admits a basis $\{e, x_1,
x_2, \dots, x_p, y_1, y_2, \dots, y_q, z_1, z_2, \dots, z_r\}$
such that
$$P_1 = \{e, x_1, x_2, \dots, x_p,\} \quad P_0 = \{y_1, y_2,
\dots, y_q\}, \quad P_{\frac 1 2} = \{z_1, z_2, \dots, z_r\}.$$

The assertion of Theorem \ref{th Jacobson} provide the table of
multiplication in this basis:
$$J:\left\{\begin{array}{lll} e\cdot x_i=x_i, & x_i \cdot x_j= \alpha_{i,j}e + \sum\limits_{k=1}^p \beta_{i,j}^k x_k,
& x_i \cdot z_j=  \sum\limits_{k=1}^r \delta_{i,j}^k z_k, \\[1mm]
 & y_i \cdot y_j= \sum\limits_{k=1}^q \gamma_{i,j}^k
y_k, & y_i \cdot z_j= \sum\limits_{k=1}^r \nu_{i,j}^k z_k, \\[1mm]
e\cdot z_i=\frac 1 2 z_i, & z_i \cdot z_j= \lambda_{i,j}e +
\sum\limits_{k=1}^p \mu_{i,j}^k x_k+
\sum\limits_{k=1}^q\theta_{i,j}^k y_k.
\end{array}\right.$$

It is easy to see that condition $p=q=0$ implies the
multiplication
$$e\cdot e=e, \quad e\cdot z_i=\frac 1 2 z_i, \quad z_i \cdot z_j= \lambda_{i,j}e.$$
From Jordan identity we get $\lambda_{i,j}=0$ and the
algebra $\nu_n(\frac 1 2)$ is obtained. However, this algebra is an algebra of level one.

Therefore, we assume that $(p,q) \neq (0,0).$ Taking the
degeneration \[g_t: \ g_t(e)= e, \ g_t(x_i)=t^{-1}x_i, \
g_t(y_j)=t^{-1}y_j, \ g_t(z_k)=t^{-1}z_k,\] we easily obtain that
any Jordan algebra $J$ with condition of non-triviality of
semi-simple part is degenerates to the following algebra
$$e\cdot e=e, \  e\cdot x_i=x_i,\  e\cdot y_j=0, \  e\cdot z_k=\frac 1 2 z_k,
\quad 1 \leq i \leq p, \ 1 \leq j \leq q, \ 1 \leq k \leq r.$$

\begin{itemize}
\item If $p=r=0,$ then we obtain the algebra $J_1.$
\item If $q=r=0,$ then we get the algebra $J_2.$
\item If two of the $p, \ q, \ r$ are non-zero, then denoting by $e_1=e$ and elements $\{x_i, y_j, z_k\}$
by elements $\{e_i\}, 2 \leq i \leq n,$ we rewrite the table of
multiplication as follows:
$$ J(\zeta_i) : \ e_1 \cdot e_1=e_1, \quad  e_1 \cdot
e_{i}= \zeta_i e_i, \quad 2 \leq i \leq n,$$ where $\zeta_i \in
\{0; \frac 1 2; 1\}$ and there exist $i,$ $j$ such that $\zeta_i
\neq \zeta_j.$ Without loss of generality, one can suppose $\zeta_2
\neq \zeta_3.$ Taking the degeneration $g_t$ defined as
\[g_t^{-1}:\left\{\begin{array}{lll} g_t^{-1}(e_1)= t e_1, & g_t^{-1}(e_2)=e_2+e_3,
\\ g_t^{-1}(e_3)=t(\zeta_2 e_2+\zeta_3e_3), & g_t^{-1}(e_{i})=e_{i},
& 4 \leq i \leq n, \end{array}\right.\] we obtain that algebra
$J(\zeta_i)$ degenerates to $J_3.$
\end{itemize}

Now we consider case  of $J_{ss} =0,$ i.e., Jordan algebra $J$ is
nilpotent.

\textbf{Case 1.} Let $dim J^2 \geq 2.$ Then algebra $J$ admits a
basis $\{x_1, x_2,  \dots, x_n\}$ such that $\{x_1, x_2, \dots,
x_{k}\} \in J \setminus J^2$ and $x_{k+1}, x_{k+2} \in J^2.$
Moreover, one can assume $x_1 \in J \setminus J^2$ and $x_1\cdot
x_1 =x_{k+1} \in  J^2\setminus J^3.$

\textbf{Case 1.1.} Let $dim (J^2/ J^3) \geq 2.$ Then $x_{k+2} \in
J^2 \setminus J^3$ and we can suppose  $x_1 \cdot x_2 = x_{k+2}.$

Indeed, if there exists some $i $ such that  $x_1 \cdot x_i \notin
span\langle x_{k+1}\rangle,$ then without loss of generality, we
can suppose $i=2$ and derive $x_1 \cdot x_2 = x_{k+2}.$

Let now $x_1 \cdot x_i \in span\langle x_{k+1}\rangle$ for any
$i.$ We set $x_1\cdot x_i  = \alpha_i x_{k+1}, \  2 \leq i \leq
k.$ The condition $x_{k+2} \in J^2 \setminus J^3$ implies the
existence of $j, \ 2 \leq j \leq k$ such that $x_j\cdot x_j =
x_{k+2}.$ Without loss of generality, one can assume $j=2.$ Hence,
we obtain the products
$$x_1 \cdot x_1 = x_{k+1}, \quad x_1\cdot x_2 =\alpha_2 x_{k+1}, \quad x_2\cdot x_2 = x_{k+2}.$$

Taking the change of basis
$$x_1' = x_1 + A x_2, \ x_2'=x_2, \ x_{k+1}' = (1+2A\alpha_2)x_{k+1} + A^2 x_{k+2}, \ x'_{k+2}= \alpha_2 x_{k+1} + Ax_{k+2},$$
where $A(1+A\alpha_2) \neq 0,$ we derive
$$x_1'\cdot x_1' = x'_{k+1},\quad
x_1'\cdot x_2' =  x'_{k+2}.$$

Therefore, in this subcase we have shown that there exists a
basis $\{x_1,x_2, \dots, x_{k+1}, x_{k+2} \dots, x_n\}$ such that
$$x_1\cdot x_1 = x_{k+1}, \quad x_1\cdot x_2 = x_{k+2}, \quad x_2\cdot x_2 = \gamma_{k+1}x_{k+1}+ \gamma_{k+2}x_{k+2}+ \dots +
\gamma_{n}x_{n}.$$

Taking the degeneration
\[g_t: \left\{\begin{array}{lll} g_t(x_1)= t^{-2} x_1,&
 g_t(x_2)= t^{-2}x_2,\\
 g_t(x_{k+2})=t^{-4}x_{k+2},& g_t(x_{i})=t^{-3}x_{i},  &  i \neq k+2, \ 3 \leq i \leq
n,\end{array}\right. \] we obtain that algebra $J$ degenerates to
algebra with the following table of multiplication:
$$x_1 \cdot x_2 = x_{k+2},\quad x_2 \cdot x_2 = \gamma_{k+2}x_{k+2}.$$

Obviously, this algebra is isomorphic to algebra $J_3$ (by the
basis transformation $x_2':= x_2 - \gamma_{k+2}x_1$ and
$x'_i:=x_i$ for $i\neq 2$).

\

\textbf{Case 1.2.} Let $dim (J^2/ J^3) = 1.$ Then $x_{k+2} \in
 J^3.$ If there exist $i, \ j$ such that $x_1 \cdot
x_i \notin span\langle x_{k+1}\rangle$ or $x_i \cdot x_j \notin
span\langle x_{k+1}\rangle,$ then similarly to Case 1.1 we
conclude that algebra $J$ degenerates to algebra $J_3.$  Now we
consider the case of $x_1 \cdot x_i, \ x_i \cdot x_j \in
span\langle x_{k+1}\rangle.$

We set $$x_1 \cdot x_i = \alpha_{1,i}x_{k+1}, \quad x_i \cdot x_j
\in \alpha_{i,j}x_{k+1}, \qquad 2 \leq i,j \leq k.$$

Due to $x_{k+2} \in J^3,$ we get existence of some $i_0 \ (1 \leq
i_0 \leq k+1)$ such that $x_{i_0} \cdot x_{k+1}= x_{k+2}.$ Without
loss of generality, we can assume $i_0=1.$ Indeed, if $x_1 \cdot
x_{k+1} = 0,$ then taking the change
$$x_1' = x_1 + Ax_{i_0}, \quad x_{k+1}'=(1+2A\alpha_{1,i_0}+A^2\alpha_{i_0,
i_0})x_{k+1}, \quad  x_{k+2}'=(1+2A\alpha_{1,i_0}+A^2\alpha_{i_0,
i_0})A x_{i_0}x_{k+1},
$$ we obtain
$$x_1\cdot x_1 = x_{k+1}, \quad x_1\cdot x_{k+1} = x_{k+2}, \quad x_{k+1}\cdot x_{k+1} = \gamma_{k+2}x_{k+2}+ \dots + \gamma_{n}x_{n}.$$

Taking the degeneration
\[g_t: \left\{\begin{array}{lll} g_t(x_1)= t^{-2} x_1,& g_t(x_{i})=t^{-3}x_{i},  &   2 \leq i \leq
n, \quad i \neq k+1; k+2,\\
 g_t(x_{k+1})= t^{-2}x_{k+1},&
 g_t(x_{k+2})=t^{-4}x_{k+2}, \end{array}\right. \]
we conclude that the algebra $J$ degenerates to algebra with the
following table of multiplications:
$$x_1 \cdot x_{k+1} = x_{k+2},\quad x_{k+1} \cdot x_{k+1} = \gamma_{k+2}x_{k+2},$$
which is isomorphic to $J_3.$

\textbf{Case 2.} Let $dim J^2 = 1.$ Then $J^3=0$ and either $J$ has a three-dimensional indecomposable subalgebra $\widetilde{J}$  with conditions $dim \widetilde{J}^2 =1, \ \widetilde{J}^3=0$ or $J$ is isomorphic to the algebra $\lambda_2 \oplus \mathbb{C}^{n-2}.$ Taking into account that $J$ is not isomorphic to $\lambda_2 \oplus \mathbb{C}^{n-2}$  and that any three-dimensional indecomposable
Jordan algebra satisfying to above conditions is isomorphic to the algebra: $x_1 \cdot x_2 =x_3$ (in denotation of \cite{Kashuba2} this algebra is $T_4$), we conclude that Jordan algebra $J$ admits a basis
$\{x_1, x_2, \dots, x_n\}$ such that the table of multiplications in this basis is as follows:
 $$x_1\cdot x_2 = x_n, \quad x_1\cdot x_i = \alpha_{i}x_n, \quad x_2\cdot x_i =\beta_{i}x_n,
 \quad x_j\cdot x_i =\gamma_{i,j}x_n, \quad 2 \leq i,j \leq n.$$

Taking the following degeneration
\[g_t: \quad g_t(x_1)= x_1, \quad g_t(x_2)= x_2, \quad g_t(x_{n})=x_{n},\quad
 g_t(x_{i})=t^{-1}x_{i},  \quad \ 3 \leq i \leq n-1,\]
we obtain that algebra $J$ degenerates to $J_3.$

In order to complete the proof of theorem we need to establish that algebras $J_1,$ $J_2$ and $J_3$ do not degenerate to each other. For this purpose we shall apply invariant argumentations.

Due to nilpotency of $J_3$ we have $J_1, J_2 \notin \overline{Orb(J_3)}.$
Computing of dimensions of the spaces of derivations we get $$dim( Der(J_1))=n^2-2n+1,\quad
dim( Der(J_2))=n^2-2n+1, \quad dim(Der(J_3 )) = n^2-3n+4.
$$
Since $dim( Der(J_1))=dim( Der(J_2)) \geq dim(Der(J_3 ))$  we obtain that
$J_2, J_3 \notin \overline{Orb(J_1)}$ and $J_1,
J_3 \notin \overline{Orb(J_2)}.$
\end{proof}

\begin{rem}
Note that in the variety of $2$-dimensional Jordan algebras the
algebras of level two are $J_1$ and $J_2.$
\end{rem}

\subsection{Lie algebras of level two}

In this subsection we will describe algebras of level two in the varieties of
complex $n$-dimensional Lie and associative algebras.

We denote by $Lie_n(\mathbb{C})$ the variety of $n$-dimensional complex Lie algebras.

Thanks to work \cite{B1} we have the lists of algebras of
level two in the varieties $Lie_3(\mathbb{C})$ and $Lie_4(\mathbb{C})$. Namely, we can state the next proposition.

\begin{prop}
Algebras of level two of the variety $Lie_3(\mathbb{C})$ up to
isomorphism are the following:
$$\begin{array}{rll}
r_2 \oplus \mathfrak{a}_1: &  [e_1, e_2] = e_2, \\ r_3(\alpha): &
[e_1, e_2] = e_2, &  [e_1, e_3] = \alpha e_3, \ |\alpha|<1, \ or \
\alpha=\pm 1.
\end{array}$$
Algebras of level two of the variety $Lie_4(\mathbb{C})$ up to
isomorphism are the following:
$$\begin{array}{rlll}
n_4: & [e_1,e_2]= e_3, &  [e_1,e_3] = e_4, \\
r_2 \oplus \mathfrak{a}_2: &  [e_1, e_2] = e_2,\\
r_3(1)\oplus \mathfrak{a}_1: & [e_1, e_2] = e_2, & [e_1, e_3] =
e_3, \\
g_{4,1}(\alpha): & [e_1, e_2]= \alpha e_2, & [e_1, e_3] = e_3,
&[e_1, e_4] = e_4, \ \alpha \neq 1, \ \alpha \in \mathbb{C}^{*},
\\
g_{4,2}: & [e_1, e_2]=
 e_2 + e_3, & [e_1, e_3] = e_3,
& [e_1, e_4] = e_4.
\end{array}$$
\end{prop}

We consider Lie algebras
$$\begin{array}{rlll}
n_{5,1}\oplus \mathfrak{a}_{n-5}: & [e_1, e_3] = e_5, & [e_2, e_4]
=
e_5, \\
n_{5,2}\oplus \mathfrak{a}_{n-5}: & [e_1, e_2] = e_4, & [e_1,
e_3] = e_5, \\
r_2 \oplus \mathfrak{a}_{n-2}: &  [e_1, e_2] = e_2,\\
g_{n,1}(\alpha): & [e_1, e_2]= \alpha e_2, & [e_1, e_i] = e_i, & 3
\leq i \leq n, \ \alpha \neq 1, \ \alpha \in \mathbb{C}^{*}
\\
g_{n,2}: & [e_1, e_2]=
 e_2 + e_3, & [e_1, e_i] = e_i, & 3
\leq i \leq n.
\end{array}$$

Further we shall need the following lemma.

\begin{lem} \label{lemInvariant}
$$\begin{array}{rclcrcl}
dim( Der(n_{5,1}\oplus \mathfrak{a}_{n-5}))&=&n^2-5n+15, & & dim(
ab(n_{5,1}\oplus \mathfrak{a}_{n-5}))&=&n-2, \\[1mm]
dim(Der(n_{5,2}\oplus \mathfrak{a}_{n-5} )) &=& n^2-5n+13, & &
dim(
ab(n_{5,2}\oplus \mathfrak{a}_{n-5}))&=&n-1, \\[1mm]
dim(Der(r_2 \oplus \mathfrak{a}_{n-2}))&=&n^2-3n+4, & & dim(ab(r_2
\oplus \mathfrak{a}_{n-2}))&=&n-1, \\[1mm]
dim(Der(g_{n,1}(\alpha)))&=&n^2-3n+4, & & dim(ab(g_{n,1}(\alpha)))&=&n-1, \\[1mm]
dim(Der(g_{n,2}))&=&n^2-3n+4, & & dim(ab(g_{n,2}))&=&n-1,
\end{array}$$
where $ab(G)$ is a maximal abelian ideal of $G.$

\end{lem}
In the following theorem we present a complete list of algebras of level two in the variety $Lie_n(\mathbb{C}), \ n \geq 5$.

\begin{thm}\label{thLie} An arbitrary $n \ (n \geq 5)$-dimensional Lie algebra of level
two is isomorphic to one of the following pairwise non-isomorphic algebras:
$$n_{5,1}\oplus \mathfrak{a}_{n-5}, \quad n_{5,2}\oplus \mathfrak{a}_{n-5},
\quad r_2 \oplus \mathfrak{a}_{n-2} \quad g_{n,1}(\alpha), \quad g_{n,2}.$$
\end{thm}

\begin{proof}
\textbf{I.}  Firstly, we consider where $G$ is a nilpotent algebra. We distinguish the following cases.

\textbf{Case 1.} Let $Dim G^2 =1.$ Then $G$ is isomorphic to either Heisenberg algebra
 $H_{n=2k+1}$ or $H_{2k+1} \oplus \mathfrak{a}_{n-2k-1}.$ Thus, there exists a basis $\{x_1, x_2,
\dots, x_k, y_1, y_2, \dots, y_k, z, p_{1}, \dots, p_{n-2k-1}\}$
of $G$ such that $[x_i, y_i] = z, \ 1\leq i \leq k.$

Clearly, $k\geq 2$ because otherwise $G$ is an algebra of level one. Taking the degeneration
\[g_t: \left\{\begin{array}{llll} g_t(x_1)=  x_1,& g_t(x_2)=  x_2,& g_t(x_{i})=t^{-1}x_{i},  &  3 \leq i \leq
k,\\
g_t(y_1)=  y_1,& g_t(y_2)=  y_2,&
g_t(y_{i})=t^{-1}y_{i},  &  3 \leq i \leq k,\\
g_t(z)=  z, \end{array}\right.
\]
we obtain that algebra $G$ degenerates to $n_{5,1}\oplus
\mathfrak{a}_{n-5}.$

\textbf{Case 2.} Let $Dim G^2 \geq 2.$ We suppose that $\{x_1, x_2, \dots,
x_k\}$ are generator basis elements of $G.$ Then, without loss of generality, we can
assume $[x_1, x_2] = x_{k+1}.$

Below we show that it may always be assumed $$[x_1, x_2]
= x_4, \ [x_1, x_3] = x_5.$$

\begin{itemize}
\item Let there exists $i_0$ such that $[x_1, x_{i_0}] \notin
span\langle x_{k+1}\rangle,$ then taking
$$x'_1= x_1, \  x'_2 = x_2, \ x'_3 = x_{i_0}, \ x'_4 = x_{k+1}, \ x'_5 = [x_1, x_{i_0}] $$
we obtain $[x'_1, x'_2] = x'_4, \ [x'_1, x'_3] = x'_5.$

\item Let $[x_1, x_{i}] \in span\langle x_{k+1}\rangle$ for all $3 \leq i \leq k$ and
there exists some $i_0$ such that $[x_2, x_{i_0}] \notin
span\langle x_{k+1}\rangle.$ Due to symmetricity of $x_1$ and
$x_2,$ similarly to the previous case we can choose a basis
$\{x'_1, x'_2, \dots, x'_n\}$ with condition $[x'_1, x'_2] = x'_4,
\ [x'_1, x'_3] = x'_5.$

\item Let $[x_1, x_{i}],\  [x_2, x_{i}] \in span\langle x_{k+1}\rangle$ for all $3 \leq i \leq
k.$ We set $[x_1, x_{i}] = \alpha_{i}x_{k+1}$ and $[x_2, x_{i}] =
\beta_{i}x_{k+1}.$ Let $x_{i_0}$ and $x_{j_0}$ are generators of
$G$ such that $[x_{i_0}, x_{j_0}] \notin span\langle
x_{k+1}\rangle.$ Since $dim G^2 \geq 2$ one can assume $[x_{i_0},
x_{j_0}] = x_{k+2}.$

Putting $$x'_1 = x_1 + Ax_{i_0},  \ x'_2= x_2,  \ x'_3 = x_{j_0},
\ x'_{4} = (1-A\beta_{i_0})x_{k+1}, \ x'_5 = Ax_{k+2} +
\alpha_{i_0} x_{k+1}$$ with $A(1-A\beta_{i_0}) \neq 0,$ we deduce
$[x'_1, x'_2] = x'_4, \ [x'_1, x'_3] = x'_5.$

\item Let $[x_i, x_{j}] \in span\langle x_{k+1}\rangle$ for all $1 \leq i,j \leq
k.$ Then for some $i_0$ we have $[x_{i_0}, x_{k+1}] \neq 0.$
Without loss of generality, one can assume $[x_1, x_{k+1}] =
x_{k+2}.$
    \begin{itemize} \item If $k\geq 3,$ then setting
    $$x'_1 = x_1,  \ x'_2= x_2,  \ x'_3 =
    x_{3}+x_{k+1}, \ x'_{4} = x_{k+1}, \ x'_5 = x_{k+2} + \alpha_{1,3} x_{k+1},$$ we obtain $[x'_1, x'_2] = x'_4, \ [x'_1, x'_3] =
x'_5.$

    \item If $k=2,$ then we have $[x_1,x_2]=x_3, \ [x_1,x_3]=x_4.$
It is not difficult to obtain that
    $[x_1,x_4]=x_5$ or $[x_2,x_3]=x_5$ (because of $n \geq 5$). Taking $$x'_1 = x_1,  \ x'_2= x_2,  \ x'_3 =    x_{4}, \ x'_{4} = x_{3}, \ x'_5 = x_{5}$$ in the case of $[x_1,x_4]=x_5$ and
    $$x'_1 = -x_3,  \ x'_2= x_1,  \ x'_3 =x_{4}, \ x'_{4} = x_{2}, \ x'_5 = x_{5}$$ in the case of $[x_2,x_3]=x_5,$ we derive the products $[x_1, x_2] = x_4, \ [x_1, x_3] = x_5.$

    \end{itemize}
\end{itemize}

Thus, there exists a basis $\{x_1, x_2, x_3, \dots, x_n\}$ of $G$
with the products $$[x_1, x_2] = x_4, \ [x_1, x_3] = x_5.$$

Note that $G$ degenerates to the algebra with multiplication:
$$[x_1, x_2] = x_4, \ [x_1, x_3] = x_5, \ [x_2, x_3] = \gamma_4 x_4 + \gamma_4 x_5$$
via the following degeneration:
\[g_t: \left\{\begin{array}{llll} g_t(x_1)= t^{-2} x_1,& g_t(x_2)= t^{-2} x_2,& g_t(x_3)= t^{-2}
x_3,&\\
g_t(x_4)= t^{-4} x_4,& g_t(x_5)= t^{-4} x_5,&
g_t(x_{i})=t^{-3}x_{i},  & 6 \leq i \leq n. \end{array}\right. \]

From the change of basis $x_2'=x_2 - \gamma_5x_1,$ $x_3'=x_3 +
\gamma_4x_1,$ we obtain that this algebra is isomorphic to
$n_{5,2} \oplus \mathfrak{a}_{n-5}.$

\textbf{II.}  Let $G$ be a solvable Lie algebra with nilradical
$N.$ Since the nilradical $N$ degenerates to the abelian algebra,
we conclude that any solvable Lie algebra degenerates to the
solvable algebra with abelian nilradical. Therefore, one can
assume that $G$ is a solvable Lie algebra with abelian nilradical.

Moreover, if $codim N \geq 2,$ then choosing
a basis $\{x_1, x_2, x_3, \dots, x_n\}$ such that $\{x_1, x_2, \dots, x_k\}$ is a basis of complementary space to $N$ and taking the degeneration
$$g_t(x_1) = x_1, \ g_t(x_2) = t^{-1}x_2, \dots,  \ g_t(x_k) = t^{-1}x_k, \ g_t(x_{k+1}) = x_{k+1},
  \dots, g_t(x_{n}) = x_{n},$$
we obtain that $G$ degenerates to a solvable Lie algebra with
nilradical of codimension equal to 1.

Therefore, we assume that algebra $G$ admits a basis $\{x_1, x_2,
\dots, x_{n}\}$ with nilradical $N = \{x_2, x_3, \dots, x_{n}\}$
and restriction of the operator $ad({x_1})$ on $N$ has a Jordan
normal form $ad({x_1})_{| N} = (J_{k_1}, J_{k_2}, \dots J_{k_s}).$

It is easy to see that, if the operator $ad({x_1})_{| N}$ is a scalar matrix, that is, $ad({x_1})_{| N}$ has a unique eigenvalue and $k_i=1$ for all $i \ (1 \leq i \leq s),$ then $G$ is
an algebra of level one (namely, $G \cong p_n$).

Let operator $ad({x_1})_{| N}$ has a unique eigenvalue, but there exists a
Jordan block of order greater than one. One can assume
$k_1\geq 2.$ Taking the degeneration
\[g_t: \left\{\begin{array}{llll} g_t(x_1)= x_1,& g_t(x_2)= t^{2-k_1} x_2,\\ g_t(x_i)= t^{i-1-k_1}
x_i,& 3 \leq i \leq k_1+1,\\
g_t(x_{k_1+\dots+k_{j-1}+i})= t^{i-1-k_j}
x_{k_1+\dots+k_{j-1}+i},& 2 \leq j \leq s, & 2 \leq i \leq k_j+1,
\end{array}\right. \]
we obtain that $G$ degenerates to the algebra $g_{n,2}.$

Let operator $ad({x_1})_{| N}$ has different eigenvalues. Taking the following degeneration:
\[g_t: g_t(x_1)= x_1, \
g_t(x_{k_1+\dots+k_{j-1}+i})= t^{i-1-k_j}
x_{k_1+\dots+k_{j-1}+i},\quad 1 \leq j \leq s, \ 2 \leq i \leq
k_j+1, \] we conclude that algebra $G$ degenerates to an algebra of the family:
\[\begin{array}{rlll}  g_{n,1}(\alpha_3, \dots \alpha_{n}): & [x_1,
x_2]= x_2, &  [x_1, x_i] = \alpha_i x_i, & 3 \leq i \leq n, \
(\alpha_3, \dots \alpha_{n}) \neq (1, \dots, 1).\end{array}\]

Note that $g_{n,1}(0,0, \dots 0)$ is isomorphic to the algebra $r_2 \oplus
\mathfrak{a}_{n-2}$ and algebras $g_{n,1}(1, \dots 1, \alpha_j,
1, \dots 1)$ with $\alpha_j \neq 1$ and $g_{n,1}(\alpha, \alpha,
\dots \alpha)$ with $\alpha \neq 1$ are isomorphic to
$g_{n,1}(\alpha).$

For the rest cases of parameters $\alpha_i$ we can assume that $\alpha_3 \neq 1$ and $\alpha_4 \neq
\alpha_5.$

Making the basis transformation
$$e_1 = x_1, \ e_2 = x_2+ x_3, \ e_3 = x_2+\alpha_3 e_3, \ e_4 = x_4+x_5,  \ e_5 = \alpha_4 x_4+ \alpha_5 x_5,
 \ e_i = x_i, \quad 6 \leq i \leq n,$$
we get the multiplication
\[\begin{array}{llll} [e_1,e_2] = e_3, & [e_1,e_3] = -\alpha_3 e_2 +
(1+\alpha_3)e_3, \\[1mm]  [e_1,e_4] = e_5, & [e_1,e_5] =
 -\alpha_4\alpha_5 e_4 + (\alpha_4 +\alpha_5)e_5, & [e_1,e_i] = e_i,
& 6 \leq i \leq n.\end{array}\]

Similarly to the nilpotent case,  $G$ degenerates to the algebra
$n_{5,2} \oplus \mathfrak{a}_{n-5}$ via degeneration
\[g_t: \left\{\begin{array}{llll}g_t(x_1) = t^{-1}x_1, & g_t(x_2) = t^{-1}x_2, &  g_t(x_3) = t^{-2}x_3, \\[1mm]
g_t(x_4) = t^{-1}x_4, & g_t(x_4) = t^{-2}x_5, &
 g_t(x_i) = x_i,  & 6 \leq i \leq n.\end{array}\right.\]

\textbf{III.} Let $G$ has not-trivial semi-simple part. Due to
Levi's decomposition we have $G=(S_1\oplus \dots \oplus
S_k)\dot{+}R,$ where $S_i$ are simple Lie ideals and $R$ is
solvable radical. From the classical theory of Lie algebras
\cite{Jac1} we know that any simple Lie algebra $S$ has root
decomposition with respect to regular element $x$. Namely we have
$$S=S_0 \oplus S_{\alpha}  \oplus S_{-\alpha}\oplus S_{\beta}  \oplus S_{-\beta}\oplus \dots \oplus
 S_{\gamma}  \oplus S_{-\gamma}, \ x\in S_0.$$

Let $\{x_1, x_2, \dots, x_n\}$ be a basis such that $x_1=x, \ x_2\in S_{\alpha}$ and $x_3\in S_{-\alpha}$ with $\alpha\neq0$. Then $[x_1, x_2]= \alpha x_2$ and $[x_1, x_3] = - \alpha x_3.$ By scaling of basis elements we can assume that $\alpha = 1.$

Taking the degeneration $$g_t(x_1) = x_1, \ g_t(x_i) = t^{-1}x_i,
\ 2 \leq i \leq n,$$ we obtain that $G$ is degenerated to the following algebra:
$$[x_1, x_2] = x_2, \quad [x_1, x_3] = - x_3, \quad [x_1, x_i] \in lin<x_2, x_3, \dots, x_n>.$$ Obviously, this
solvable algebra is not an algebra of level one (From Case
\textbf{II}).

Hence, any Lie algebra $G$ with non-trivial semi-simple part has
not level two.

Thus, we have proved that any Lie algebra, which is not level one,
degenerates to one of the algebras:
$$ n_{5,1}\oplus \mathfrak{a}_{n-5}, \quad n_{5,2}\oplus
\mathfrak{a}_{n-5}, \quad r_2 \oplus \mathfrak{a}_{n-2}, \quad
g_{n,1}(\alpha), \quad g_{n,2}.$$

Taking into account that the property $\lambda \rightarrow \mu$ implies $dim Der(\lambda) < dim Der(\mu)$ and $dim ab(\lambda) < dim ab(\mu)$ and Lemma \ref{lemInvariant} we conclude that these
algebras do not degenerate to each other.
\end{proof}

Applying similar techniques as in the proof of Theorems \ref{thJordan} and  \ref{thLie} we obtain
the list of $n$-dimensional associative algebras of level two.

\begin{thm} Any $n$-dimensional associative algebra of level two
is isomorphic to one of the following algebras:
$$\begin{array}{rllll}A_1: & e \cdot e=e;\\[1mm]
A_2: & e_1 \cdot e_1=e_1, & e_1 \cdot
e_{i}=e_i, &  e_i \cdot e_1 = e_i, & 2 \leq i \leq n;\\[1mm]
A_3: & e_1 \cdot e_1=e_1, & e_1 \cdot
e_{i}= e_i, &  & 2 \leq i \leq n;\\[1mm]
A_4: & e_1 \cdot e_1=e_1, &  e_i \cdot e_1 = e_i, & & 2 \leq i \leq n;\\[1mm]
A_5(\alpha) :&
e_2\cdot e_1=e_3,&
 e_1\cdot e_2=\alpha e_3, & & \alpha \neq \alpha^{-1};\\[1mm]
A_6: &  e_1\cdot e_1=e_3, & e_2\cdot e_1= e_3, & e_1\cdot e_2=-e_3.\\
 \end{array}$$
\end{thm}

\textbf{Acknowledgment.} The author would like to gratefully
acknowledge for support to IMU/CDC-program.


\begin{thebibliography}{9}

\bibitem{B1} Burde D. {\it Degenerations of 7-dimensional nilpotent Lie algebras,} Comm.
Algebra, 33(4)  (2005) 1259--1277.

\bibitem{B2} Burde D., Steinhoff C. {\it Classification of orbit closures $4$-dimensional complex Lie
algebras,} J. Algebra, 214(2) (1999) 729--739.

\bibitem{Casas} Casas J.M., Khudoyberdiyev A.Kh., Ladra M., Omirov B.A.
{\it On the degenerations of solvable Leibniz algebras,} Linear
Alg. Appl.,  439(2) (2013) 472--487.


\bibitem{Gorb1} Gorbatsevich V.V.
{\it On contractions and degeneracy of finite-dimensional
algebras,} Soviet  Math.  (Iz. VUZ)  35  (10) (1991) 17-–24.

\bibitem{Gorb2}   Gorbatsevich V.V. {\it Anticommutative
finite-dimensional algebras of the first three levels of
complexity,} St. Petersburg Math. J. 5 (1994) 505–-521.

\bibitem{GrunHall} Grunewald F., O'Halloran J. {\it Varieties of nilpotent Lie algebras of dimension less than six,}
 J. Algebra, 112 (2) (1998) 315--325.

\bibitem{Jac1} {\rm Jacobson N., }{\it Lie algebras,} Interscience Publishers, Wiley, New York, 1962.

\bibitem{Jacobson} Jacobson N. {\it Structure and representations of Jordan algebras,}
Amer. Math. Soc. Colloq. Publ., 39 (1968).

\bibitem{Kashuba1}  Iryna Kashuba, Mar\'{i}a Eugenia Martin. {\it The variety of three-dimensional real Jordan
algebras,} arXiv:1404.5001v1, 16 p.

\bibitem{Kashuba2}  Iryna Kashuba, Mar\'{i}a Eugenia Martin. {\it Deformations of
Jordan algebras of dimension four,} J. Algebra, 399 (2014)
277–-289.

\bibitem{Khud}  Khudoyberdiyev A.Kh., Omirov B.A. {\it The classification of
algebras of level one,} Linear Alg.  Appl., 439 (11) (2013)
3460--3463.


\bibitem{Seeley} Seeley C., {\it Degeneration of $6$-dimensional nilpotent Lie algebras over
$\mathbb{C}$,} Comm. Algebra,  18(10) (1990) 3493--3505.

\end{thebibliography}
\end{document}